\documentclass[preprint,12pt]{elsarticle}
\usepackage{amsthm,amsfonts,amssymb,amscd,amsmath,enumerate,verbatim,calc,graphicx,geometry}
\usepackage[all]{xy}
\newtheorem{theorem}{Theorem}[section]
\newtheorem{lemma}[theorem]{Lemma}

\newtheorem{corollary}[theorem]{Corollary}
\theoremstyle{definition}
\theoremstyle{definitions}
\newtheorem{definition}[theorem]{Definition}

\newtheorem{remark}[theorem]{Remark}
\newtheorem{example}[theorem]{Example}
\theoremstyle{notations}

\theoremstyle{remarks}

\journal{}

\begin{document}

\begin{frontmatter}



\title{Adjointness of Suspension and Shape Path Functors}


\author[label2]{Tayyebe~Nasri}
\ead{tnasri72@yahoo.com}
\author[label1]{Behrooz~Mashayekhy}
\ead{bmashf@um.ac.ir}
\author[label1]{Hanieh~Mirebrahimi\corref{cor1}}
\ead{h\_mirebrahimi@um.ac.ir}

\address[label1]{Department of Pure Mathematics, Center of Excellence in Analysis on Algebraic Structures, Ferdowsi University of Mashhad,\\
P.O.Box 1159-91775, Mashhad, Iran.}
\address[label2]{Department of Pure Mathematics, Faculty of Basic Sciences,  University of Bojnord,\\
 Bojnord, Iran.}
\cortext[cor1]{Corresponding author}

\begin{abstract}
In this paper, we introduce a subcategory $\widetilde{Sh}_*$ of Sh$_*$ and obtain some results in this subcategory. First we show that there is a natural bijection $Sh (\Sigma (X, x), (Y,y))\cong Sh((X,x),Sh((I, \dot{I}),(Y,y)))$, for every $(Y,y)\in \widetilde{Sh}_*$ and $(X,x)\in Sh_*$. By this fact, we prove that for any pointed topological space $(X,x)$ in $\widetilde{Sh}_*$, $\check{\pi}_n^{top}(X,x)\cong \check{\pi}_{n-k}^{top}(Sh((S^k, *),(X,x)), e_x)$, for all $1\leq k \leq n-1$.
\end{abstract}

\begin{keyword}
Shape category\sep Topological shape homotopy group\sep Shape group\sep Suspensions.
\MSC[2010]  55P55\sep 55Q07\sep  54H11\sep 55P40

\end{keyword}

\end{frontmatter}

\section{Introduction and Motivation}
Mor\'on et al. \cite{M1} gave a complete, non-Archimedean metric (or ultrametric) on the set of shape morphisms between two unpointed compacta (compact metric spaces) $X$ and $Y$, $Sh(X, Y)$. They mentioned that this construction can be translated to the pointed case. Consequently, as a particular case, they obtained a complete ultrametric induces a norm on the shape groups of a compactum Y and then presented some results on these topological groups \cite{M2}. Also, Cuchillo-Ibanez et al. \cite{CM1} constructed several generalized ultrametrics in the set of shape morphisms between topological spaces and obtained semivaluations and valuations on the groups of shape equivalences and $k$th shape groups. On the other hand, Cuchillo-Ibanez et al. \cite{CM} introduced a topology on the set $Sh(X, Y)$, where $X$ and $Y$ are arbitrary topological spaces, in such a way that it extended topologically the construction given in \cite{M1}. Also, Moszy\'nska \cite{M} showed that the $k$th shape group $\check{\pi}_k(X,x)$, $k\in \Bbb{N}$, is isomorphic to the set $Sh((S^k, *),(X,x))$ consists of all shape morphisms $(S^k, *) \rightarrow (X, x)$ with a group operation, for all compact Hausdorff space $(X, x)$. Note that, Bilan \cite{B1} mentioned that this fact is true for all topological spaces.
 The authors \cite{Na} applied this topology on the set of shape morphisms between pointed spaces and proved that the $k$th shape group $\check{\pi}_k(X,x)$, $k\in \Bbb{N}$, with the above topology is a Hausdorff topological group, denoted by $\check{\pi}_k^{top}(X,x)$. In this paper, we introduce a subcategory $\widetilde{Sh}_*$ of Sh$_*$ and obtain some results in this subcategory. It is well-known that the pair $(\Sigma , \Omega)$ is an adjoint pair of functors  on hTop$_*$ and therefore, there is a natural bijection $Hom(\Sigma (X,x),(Y,y))\cong Hom((X,x), \Omega (Y,y))$, for every pointed topological spaces $(X,x)$ and $(Y,y)$. In this paper, we show that there is a natural bijection $Sh (\Sigma (X, x), (Y,y))\cong Sh((X,x),(Sh((I, \dot{I}),(Y,y)),e_y))$, for every $(Y,y)\in \widetilde{Sh}_*$ and $(X,x)\in Sh_*$.  By this fact we conclude that the functor $Sh((I, \dot{I}),-)$ preserves inverse limits such as products, pullbacks, kernels, nested intersections and completions, provided inverse limit exists in the subcategory $\widetilde{Sh}_*$. Also, the functor $\Sigma$ preserves direct limits of connected spaces in this subcategory. As a consequence, if $(X\times Y, (x,y))$ is a product of pointed spaces $(X, x)$ and $(Y,y)$ in the subcategory $\widetilde{Sh}_*$, then 
$$\check{\pi}_1(X\times Y, (x,y))\cong\check{\pi}_1(X, x)\times\check{\pi}_1(Y, y).$$
 
It is well-known that for any pointed space $(X,x)$ and for all $1\leq k\leq n-1$, $\pi_n(X,x)\cong \pi_{n-k}(\Omega(X,x), e_x)$. In this paper, we show that for any pointed topological space $(X,x)$ in $\widetilde{Sh}_*$, $\check{\pi}_n(X,x)\cong \check{\pi}_{n-k}(Sh((S^k, *),(X,x)), e_x)$, for all $1\leq k \leq n-1$. We then exhibit an example in which this result dose not hold in the category Sh$_*$. 

 Endowed with the quotient topology induced by the natural surjective map $q:\Omega^n(X,x)\rightarrow \pi_n(X,x)$, where $\Omega^n(X,x)$ is the $n$th loop space of $(X,x)$ with the compact-open topology, the familiar homotopy group $\pi_n(X,x)$ becomes a quasitopological group which is called the quasitopological $n$th homotopy group of the pointed space $(X,x)$, denoted by $\pi_n^{qtop}(X,x)$ (see \cite{B0,Br,Bra,G1}). Nasri et al. \cite{Na2}, showed that  for any pointed topological space $(X,x)$, $\pi_n^{qtop}(X,x)\cong \pi_{n-k}^{qtop}(\Omega^k(X,x), e_x)$, for all $1\leq k \leq n-1$. In this paper, we prove that for any pointed topological space $(X,x)$ in $\widetilde{Sh}_*$, $\check{\pi}_n^{top}(X,x)\cong \check{\pi}_{n-k}^{top}(Sh((S^k, *),(X,x)), e_x)$, for all $1\leq k \leq n-1$.

\section{Preliminaries}
In this section, we recall some of the main notions concerning the shape category and the pro-HTop (see \cite{MS}).
Let $\mathbf{X}=(X_{\lambda},p_{\lambda\lambda'},\Lambda)$ and $\mathbf{Y}=(Y_{\mu},q_{\mu\mu'},M)$ be two inverse systems in HTop. A {\it pro-morphism} of inverse
systems, $(f,f_{\mu}): \mathbf{X} \rightarrow \mathbf{Y} $, consists of an index function $f : M \rightarrow\Lambda$ and of mappings $f_{\mu}: X_{f(\mu)} \rightarrow Y_{\mu} $, $\mu\in M$, such that for every related pair $\mu\leq \mu'$ in $M$, there exists a $\lambda\in\Lambda$, $\lambda\geq f(\mu),f(\mu')$ so that  $$q_{\mu\mu'}f_{\mu'}p_{f(\mu')\lambda}\simeq f_{\mu}p_{f(\mu)\lambda}.$$

The {\it composition} of two pro-morphisms $(f,f_{\mu}): \mathbf{X} \rightarrow \mathbf{Y} $ and $(g,g_{\nu}): \mathbf{Y} \rightarrow \mathbf{Z}=(Z_{\nu},r_{\nu\nu'},N) $ is also a pro-morphism $ (h,h_{\nu})=(g,g_{\nu})(f,f_{\mu}): \mathbf{X} \rightarrow \mathbf{Z} $, where $h=fg$ and $h_{\nu}=g_{\nu}f_{g(\nu)}$. The {\it identity pro-morphism} on $\mathbf{X}$ is pro-morphism $(1_{\Lambda}, 1_{X_{\lambda}}): \mathbf{X} \rightarrow \mathbf{X} $, where $1_{\Lambda}$ is the identity function.
A pro-morphism $(f,f_{\mu}): \mathbf{X} \rightarrow \mathbf{Y} $ is said to be {\it equivalent} to a pro-morphism $(f',f'_{\mu}): \mathbf{X} \rightarrow \mathbf{Y} $, denoted by $(f,f_{\mu})\sim (f',f'_{\mu})$, provided every $\mu\in M$ admits a $\lambda\in\Lambda$ such that $\lambda\geq f(\mu),f'(\mu)$ and
$$f_{\mu}p_{f(\mu)\lambda}\simeq f'_{\mu}p_{f'(\mu)\lambda}.$$

The relation $\sim$ is an equivalence relation. The {\it category} pro-HTop has as objects, all inverse systems $\mathbf{X}$ in HTop and as morphisms, all equivalence classes $\mathbf{f}=[(f,f_{\mu})]$. The composition of $\mathbf{f}=[(f,f_{\mu})]$ and $\mathbf{g}=[(g,g_{\nu})]$ in pro-HTop is well defined by putting
$$\mathbf{gf}=\mathbf{h}=[(h,h_{\nu})].$$

An HPol-expansion of a topological space $X$ is a morphism $\mathbf{p} :X\rightarrow \mathbf{X}$ in pro-HTop, where $\mathbf{X}$ belongs to pro-HPol characterised by the following two properties:\\
(E1) For every $P\in HPol$ and every map $h:X\rightarrow P$ in HTop, there is a $\lambda\in \Lambda$ and a map $f:X_{\lambda}\rightarrow P$ in HPol such that $fp_{\lambda}\simeq h$.\\
(E2) If $f_0, f_1:X_{\lambda}\rightarrow P$ satisfy $f_0p_{\lambda}\simeq f_1p_{\lambda}$, then there exists a $\lambda'\geq\lambda$ such that $f_0p_{\lambda\lambda'}\simeq f_1p_{\lambda\lambda'}$.

Let $\mathbf{p} :X\rightarrow \mathbf{X}$ and $\mathbf{p'} :X\rightarrow \mathbf{X'}$ be two HPol-expansions of an space $X$ in HTop, and let $\mathbf{q} : Y \rightarrow \mathbf{Y}$ and $\mathbf{q'} : Y \rightarrow \mathbf{Y'}$ be two HPol-expansions of an space $Y$ in HTop. Then there exist two natural isomorphisms $\mathbf{i}:\mathbf{X}\rightarrow \mathbf{X}'$ and $\mathbf{j}:\mathbf{Y}\rightarrow \mathbf{Y}'$ in pro-HTop. A morphism $\mathbf{f}:\mathbf{X}\rightarrow \mathbf{Y}$ is said to be {\it equivalent} to a morphism $\mathbf{f'}:\mathbf{X'}\rightarrow \mathbf{Y'}$, denoted by $\mathbf{f}\sim\mathbf{f'}$, provided the following diagram in pro-HTop commutes:
\begin{equation}
\label{dia}\begin{CD}
\mathbf{X}@>\mathbf{i}>>\mathbf{X'}\\
@VV \mathbf{f}V@V \mathbf{f'}VV\\
\mathbf{Y}@>\mathbf{j}>>\mathbf{Y'}.
\end{CD}\end{equation}

Now, the {\it shape category} Sh is defined as follows: The objects of Sh are topological spaces. A morphism $F:X\rightarrow Y$ is the equivalence class $<\mathbf{f}>$ of a mapping $\mathbf{f}:\mathbf{X}\rightarrow \mathbf{Y}$ in pro-HTop.
The {\it composition} of $F=<\mathbf{f}>:X\rightarrow Y$ and $G=<\mathbf{g}>:Y\rightarrow Z$ is defined by the representatives, i.e., $GF=<\mathbf{g}\mathbf{f}>:X\rightarrow Z$. The {\it identity shape morphism} on a space $X$, $1_X:X\rightarrow X$, is the equivalence class $<1_\mathbf{X}>$ of the identity morphism $1_\mathbf{X}$ in pro-HTop.

Let $\mathbf{p}:X\rightarrow \mathbf{X}$ and $\mathbf{q}:Y\rightarrow \mathbf{Y}$ be HPol-expansions of $X$ and $Y$, respectively. Then for every morphism $f:X\rightarrow Y$ in HTop, there is a unique morphism $\mathbf{f}:\mathbf{X}\rightarrow \mathbf{Y}$ in pro-HTop such that the following diagram commutes in pro-HTop.
\begin{equation}
\label{dia}\begin{CD}
\mathbf{X}@<<\mathbf{p}<X\\
@VV \mathbf{f}V@V fVV\\
\mathbf{Y}@<<\mathbf{q}<Y.
\end{CD}\end{equation}

If we take other HPol-expansions $\mathbf{p'}:X\rightarrow \mathbf{X'}$ and $\mathbf{q'}:Y\rightarrow \mathbf{Y'}$, we obtain another morphism $\mathbf{f'}:\mathbf{X'}\rightarrow \mathbf{Y'}$ in pro-HTop such that $\mathbf{f'}\mathbf{p'^*}=\mathbf{q'}f$ and so we have $\mathbf{f}\sim\mathbf{f'}$. Hence every morphism $f\in HTop(X,Y)$ yields an equivalence class $<[\mathbf{f}]>$, i.e., a shape morphism $F:X\rightarrow Y$ which is denoted by $\mathcal{S}(f)$. If we put $\mathcal{S}(X)=X$ for every topological space $X$, then we obtain a functor $\mathcal{S}:HTop\rightarrow Sh$, called the {\it shape functor}. Also if $Y\in$HPol, then every shape morphism $F:X\rightarrow Y$ admits a unique morphism $f:X\rightarrow Y$ in HTop such that $F=\mathcal{S}(f)$ \cite[Theorem 1.2.4]{MS}.

Similarly, we can define the categories pro-HTop$_*$ and Sh$_*$ on pointed topological spaces (see \cite{MS}).
\section{Main Results}
In this section, we introduce a subcategory $\widetilde{Sh}_*$ of Sh$_*$ consists of all  pointed topological spaces having bi-expansions.  Then we consider the well-known suspension functor $\Sigma: Sh_*\rightarrow Sh_*$ (see \cite{MS}) and $Sh((I,\dot{I}),-): Sh_*\rightarrow Sh_*$ and show that there is a natural bijection
$Sh (\Sigma (X, x), (Y,y))\cong Sh((X,x),(Sh((I, \dot{I}),(Y,y)),e_y))$, for every $(Y,y)\in \widetilde{Sh}_*$ and $(X,x)\in Sh_*$. Then using this bijection we conclude some results in subcategory  $\widetilde{Sh}_*$. 
\begin{definition}
We say that a pointed topological space $(X,x)$ has a bi-expansion $\mathbf{p}:(X,x)\rightarrow (\mathbf{X},\mathbf{x})$ whenever $\mathbf{p}$ is an HPol$_*$-expansion of $(X,x)$ such that  $\mathbf{p_*}:Sh((I,\dot{I}),(X,x))\rightarrow \mathbf{Sh}((I,\dot{I}),(X,x))$ is an HPol$_*$-expansion of $Sh((I,\dot{I}),(X,x))$.
\end{definition}

In follow, we recall some conditions on topological space $X$ under which $X$ has a bi-expansion.
\begin{remark}\cite[Remark 4.11]{Na}.
  If $\mathbf{p}:(X,x)\rightarrow (\mathbf{X},\mathbf{x})$ is an HPol$_*$-expansion of $X$, then $\mathbf{p_*}:Sh((S^k,*),(X,x))\rightarrow \mathbf{Sh}((S^k,*),(X,x))$ is an inverse limit of $\mathbf{Sh}((S^k,*),(X,x))= (Sh((S^k,*),(X_{\lambda},x_{\lambda})), (p_{\lambda\lambda'})_*,\Lambda)$ (see \cite[Theorem 2]{CM}). Now, if $Sh((S^k,*),(X,x))$ is compact and $Sh((S^k,*),(X_{\lambda},x_{\lambda}))$ is compact polyhedron for all $\lambda\in \Lambda$, then by \cite[Remark 1]{FZ}, $\mathbf{p_*}$ is an HPol$_*$-expansion of $Sh((S^k,*),(X,x))$.
\end{remark}
\begin{lemma}\cite[Lemma 4.12]{Na}.\label{n1}
Let $(X,x)$ have an HPol$_*$-expansion $\mathbf{p}:(X,x)\rightarrow ((X_{\lambda},x_{\lambda}),p_{\lambda\lambda'},\Lambda)$ such that $\pi_k(X_{\lambda},x_{\lambda})$ is finite, for every $\lambda\in \Lambda$.  Then  \linebreak $\mathbf{p_*}:Sh((S^k,*),(X,x))\rightarrow \mathbf{Sh}((S^k,*),(X,x))$ is an HPol$_*$-expansion of \linebreak $Sh((S^k,*),(X,x))$, for all $k\in \Bbb{N}$.
\end{lemma}
\begin{example}\cite[Example 4.13]{Na} (see also \cite{MS}).
Let $\Bbb{R}P^2$ be the real projective plane. Consider the map $\bar{f}:\Bbb{R}P^2\rightarrow \Bbb{R}P^2$ induced by the following commutative diagram:
\begin{equation}
\label{dia2}\begin{CD}
D^2@<<f< D^2\\
@VV \phi V@V \phi VV\\
\Bbb{R}P^2@<<\bar{f}<\Bbb{R}P^2,
\end{CD}
\end{equation}
 where $D^2=\{z\in \Bbb{C}\ \ | \ \ |z|\leq 1\}$ is the unit 2-cell, $f(z)=z^3$ and $\phi:D^2\rightarrow \Bbb{R}P^2$ is the quotient map  identifies pairs of points $\{z,-z\}$ of $S^1$. We consider $X$ as the inverse sequence $$\Bbb{R}P^2\stackrel{\bar{f}}{\longleftarrow}\Bbb{R}P^2\stackrel{\bar{f}}{\longleftarrow}\cdots.$$
 Since $\Bbb{R}P^2 $ is compact polyhedron, by \cite[Remark 1]{FZ} $X$ is compact and $\mathbf{p}:X\rightarrow (\Bbb{R}P^2, \bar{f}, \Bbb{N})$ is an HPol-expansion  of $X$. Since $\bar{f}$ is onto and $\pi_k(\Bbb{R}P^2 )\cong \Bbb{Z}_2$ is finite, $\mathbf{p_*}:Sh((S^k,*),(X,x))\rightarrow \mathbf{Sh}((S^k,*),(X,x))$ is an HPol$_*$-expansion of\linebreak $Sh((S^k,*),(X,x))$, for all $k\in \Bbb{N}$.
\end{example}
The well-known suspension functor $\Sigma: HTop_*\rightarrow HTop_*$ is extended to a suspension functor $\Sigma: Sh_*\rightarrow Sh_*$ (see \cite{MS}). Note that, if $(X,x)$ is a pointed topological space, then $\Sigma (X,x)=(\Sigma X,\Sigma x)$ is also a pointed topological space. Therefore, whenever $\mathbf{p}:(X,x)\rightarrow (\mathbf{X},\mathbf{x})$ is an HPol$_*$-expansion of $(X,x)$, then $\mathbf{\Sigma p}:\Sigma(X,x)\rightarrow \Sigma(\mathbf{X},\mathbf{x})=(\Sigma(X_{\lambda},x_{\lambda}),\Sigma p_{\lambda\lambda'},\Lambda)$ is an HPol$_*$-expansion of $\Sigma (X,x)$.
\begin{remark}
Let $(X,x)$ be a connected topological space and $\mathbf{p}:(X,x)\rightarrow (\mathbf{X},\mathbf{x})=((X_{\lambda},x_{\lambda}), p_{\lambda\lambda'},\Lambda)$ be an HPol$_*$-expansion of $(X,x)$. Since $X$ is connected, one can assume that all $X_{\lambda}$ are connected,  by \cite[Remark 4.1.1]{MS} and so $\pi_1(\Sigma(X_{\lambda},x_{\lambda}))=0$, for all $\lambda\in \Lambda$ (by Van Kampen Theorem). Therefore, the HPol$_*$-expansion  $\mathbf{\Sigma p}:\Sigma(X,x)\rightarrow \Sigma(\mathbf{X},\mathbf{x})$ satisfies in the conditions of Lemma \ref{n1} and so $\Sigma(X,x)\in \widetilde{Sh}_*$.
\end{remark}

Let $F:\Sigma (X, x)\rightarrow (Y,y)$ be a shape morphism represented by $\mathbf{f}:\Sigma (\mathbf{X}, \mathbf{x})\rightarrow (\mathbf{Y},\mathbf{y})$ consists of $f:M\rightarrow \Lambda$ and $f_{\mu}:\Sigma (X_{f(\mu)},x_{f(\mu)})\rightarrow (Y_{\mu},y_{\mu})$. If $(Y,y)$ has a bi-expansion $\mathbf{q}:(Y,y)\rightarrow (\mathbf{Y},\mathbf{y})$, then $F$ determines a map $F^{\sharp}:(X,x)\rightarrow (Sh((I, \dot{I}),(Y,y)),e_y)$ represented by $\mathbf{f^{\sharp}}:(\mathbf{X},\mathbf{x}) \rightarrow (\mathbf{Sh}((I, \dot{I}),(Y,y)),\mathbf{e_y})$ consists of  $f:M\rightarrow \Lambda$ and $f_{\mu}^{\sharp}:(X_{f(\mu)},x_{f(\mu)})\rightarrow (Sh((I, \dot{I}),(Y_{\mu},y_{\mu})),e_{y_{\mu}})$ which is defined as $f_{\mu}^{\sharp}(x)=\mathcal{S}(l_{x\mu})$, where $l_{x\mu}:(I, \dot{I})\rightarrow (Y_{\mu},y_{\mu})$ is a map in HTop$_*$ such that $l_{x\mu}(t)=f_{\mu}([x,t])$.

In the following lemma we show that $F^{\sharp}$ is a shape morphism.
\begin{lemma}\label{lem1}
The map $F^{\sharp}$ defined in the above is a shape morphism.
\end{lemma}
\begin{proof}
With the above notation, first we show that $f_{\mu}^{\sharp}:X_{f(\mu)}\rightarrow Sh((I, \dot{I}),(Y_{\mu},y_{\mu}))$ is continuous. Since $Y_{\mu}$ is a polyhedron, the space $Sh((I, \dot{I}),(Y_{\mu},y_{\mu}))$ is discrete by \cite[Corollary 1]{CM}. Therefore, it is sufficient to show that $f_{\mu}^{\sharp}$ is locally constant. Let $x\in X_{f(\mu)}$. Since $X_{f(\mu)}$ is polyhedron, there is an open neighborhood $V_x$ of $x$ that is contractible to $x$ in $X_{f(\mu)}$. We will show that $f_{\mu}^{\sharp}$ is constant on $V_x$. Let $x'\in V_x$, then by path connectedness of $V_x$, there exists a path $\alpha:I\rightarrow X_{f(\mu)}$ such that $\alpha(0)=x$ and $\alpha(1)=x'$. We define the map $H:I\times I\rightarrow Y_{\mu}$ by $H(t,s)=f_{\mu}([\alpha(s),t])$. Since $f_{\mu}$ and $\alpha$ are continuous and $V_x$ is contractible to $x$ in $X_{f(\mu)}$, the map $H$ is well-defined and continuous. Moreover, $H$ is a relative homotopy between $f_{\mu}([x,-])$ and $f_{\mu}([x',-])$. Hence $l_{x\mu}\simeq l_{x'\mu}\ (rel\{\dot{I}\})$ and so $\mathcal{S}(l_{x\mu})=\mathcal{S}(l_{x'\mu})$. Therefore $f_{\mu}^{\sharp}(x)=f_{\mu}^{\sharp}(x')$ and so $f_{\mu}^{\sharp}$ is constant on $V_x$. Finally, we conclude that $f_{\mu}^{\sharp}$ is continuous.

Now, let $\mathbf{p}:(X,x)\rightarrow (\mathbf{X},\mathbf{x})$ be  an HPol$_*$-expansion of $(X,x)$ and $\mathbf{q}:(Y,y)\rightarrow (\mathbf{Y},\mathbf{y})$  be a bi-expansion of $(Y,y)$. The map $\mathbf{f}^{\sharp}$ is a morphism in pro-HTop$_*$.  Indeed, for any pair $\mu'\geq\mu$, there is a $\lambda\geq f(\mu), f(\mu')$ such that
\begin{equation}\label{hom1}
f_{\mu}\circ\Sigma p_{f(\mu)\lambda}\simeq q_{\mu\mu'}\circ f_{\mu'}\circ\Sigma p_{f(\mu')\lambda} \ \ (rel\{\Sigma x_{\lambda}\}).
\end{equation} 
Also, for every $x\in X_{\lambda}$,
$$f_{\mu}^{\sharp}(p_{f(\mu)\lambda}(x)) = \mathcal{S}(l_{p_{f(\mu)\lambda}(x)\mu}),$$
and for every $t\in I$,
$$l_{p_{f(\mu)\lambda}(x)\mu}(t) = f_{\mu}([p_{f(\mu)\lambda}(x), t])= f_{\mu} \circ \Sigma p_{f(\mu)\lambda}([x,t])$$
$$(q_{\mu\mu'})_*\circ l_{p_{f(\mu')\lambda}(x)\mu'}(t)= q_{\mu\mu'}\circ f_{\mu'} ([p_{f(\mu')\lambda}(x), t])= q_{\mu\mu'}\circ f_{\mu'}\circ\Sigma p_{f(\mu')\lambda} ([x,t]).$$
By (\ref{hom1}) we have $l_{p_{f(\mu)\lambda}(x)\mu} \simeq (q_{\mu\mu'})_*\circ l_{p_{f(\mu')\lambda}(x)\mu'} \ \ ( rel \{\dot{I}\})$.
Therefore
$$f_{\mu}^{\sharp}\circ p_{f(\mu)\lambda}(x)= \mathcal{S}(l_{p_{f(\mu)\lambda}(x)\mu})= \mathcal{S}((q_{\mu\mu'})_*\circ l_{p_{f(\mu')\lambda}(x)\mu'})= (q_{\mu\mu'})_*\circ f_{\mu'}^{\sharp} (p_{f(\mu')\lambda}(x)).$$
\end{proof}

On the other hand, let $G:(X,x)\rightarrow (Sh((I, \dot{I}),(Y,y)),e_y)$ be a shape morphism represented by $\mathbf{g}:(\mathbf{X},\mathbf{x}) \rightarrow (\mathbf{Sh}((I, \dot{I}),(Y,y)),\mathbf{e_y})$ consists of  $g:M\rightarrow \Lambda$ and $g_{\mu}:(X_{g(\mu)},x_{g(\mu)})\rightarrow (Sh((I, \dot{I}),(Y_{\mu},y_{\mu})),e_{y_{\mu}})$. Then we define $G^{\flat}:\Sigma (X, x)\rightarrow (Y,y)$ represented by $\mathbf{g^{\flat}}:\Sigma (\mathbf{X}, \mathbf{x})\rightarrow (\mathbf{Y},\mathbf{y})$ in pro-HTop$_*$ consists of $g:M\rightarrow \Lambda$ and $g_{\mu}^{\flat}:\Sigma (X_{g(\mu)},x_{g(\mu)})\rightarrow (Y_{\mu},y_{\mu})$ given by $g_{\mu}^{\flat}([x,t])=g'_{\mu x}(t)$, where $g'_{\mu x}$ is a unique morphism in HTop$_*$ with $\mathcal{S}(g'_{\mu x})=g_{\mu}(x)$ (see \cite[Theorem 1.2.4]{MS}).
\begin{lemma}\label{lem2}
The map $G^{\flat}$ defined in the above is a shape morphism.
\end{lemma}
\begin{proof}
First we show that $g_{\mu}^{\flat}$ is continuous. It is sufficient to show that $\overline{g_{\mu}^{\flat}}: (X_{g(\mu)}\times I,\{x_{g(\mu)}\}\times \dot{I})\rightarrow (Y_{\mu},y_{\mu})$ is continuous. We claim that the map $e_{\mu}:Sh((I, \dot{I}),(Y_{\mu},y_{\mu}))\times I\rightarrow Y_{\mu}$ given by $e_{\mu}(F,t)=F'(t)$ is continuous, where $F'$ is a unique morphism in HTop$_*$ with $\mathcal{S}(F')=F$ (see \cite[Theorem 1.2.4]{MS}). To prove the continuity of $e_{\mu}$, let $U$ be an open set containing an arbitrary point $e_{\mu}(F,t)=F'(t)$. Since $F'$ is continuous, there is an open neighbourhood $V$ of $t$ in $I$ such that $F'(V)\subseteq U$. Hence the set $\{F\}\times V$ is an open neighbourhood of $(F,t)$ in $Sh((I, \dot{I}),(Y_{\mu},y_{\mu}))\times I$ such that $e_{\mu}(\{F\}\times V)\subseteq U$. Now, the map $\overline{g_{\mu}^{\flat}}$ is equal to the composition $e_{\mu}\circ (g_{\mu}\times id)$ and so it is continuous.

Let $\mathbf{p}:(X,x)\rightarrow (\mathbf{X},\mathbf{x})$ and $\mathbf{q}:(Y,y)\rightarrow (\mathbf{Y},\mathbf{y})$ be  HPol$_*$-expansions of $(X,x)$ and $(Y,y)$, respectively. The map $\mathbf{g^{\flat}}:\Sigma (\mathbf{X}, \mathbf{x})\rightarrow (\mathbf{Y},\mathbf{y})$ is a morphism in pro-HTop$_*$. To prove this, let $\mu'\geq\mu$, then there is a $\lambda\geq g(\mu), g(\mu')$ such that
\begin{equation}\label{hom2}
(g_{\mu\mu'})_*\circ g_{\mu'} \circ p_{g(\mu')\lambda}\simeq g_{\mu}\circ p_{g(\mu)\lambda} \ \ (rel\{x_{\lambda}\}).
\end{equation} 
 Since $Y_{\mu}$ is a polyhedron, the space $Sh((I, \dot{I}),(Y_{\mu},y_{\mu}))$ is discrete by \cite[Corollary 1]{CM}. But homotopic maps in a discrete space are equal, so
 \begin{equation}\label{hom3}
(g_{\mu\mu'})_*\circ g_{\mu'} \circ p_{g(\mu')\lambda}= g_{\mu}\circ p_{g(\mu)\lambda}.
\end{equation} 
Also, for every $x\in X_{\lambda}$ and $t\in I$,
$$ g_{\mu}^{\flat} \circ \Sigma p_{g(\mu)\lambda}([x,t]) = g_{\mu}^{\flat}([p_{g(\mu)\lambda}(x), t])= g'_{\mu p_{g(\mu)\lambda}(x)}(t)$$
and
$$q_{\mu\mu'}\circ g_{\mu'}^{\flat} \circ \Sigma p_{g(\mu')\lambda}([x,t])= q_{\mu\mu'}\circ g_{\mu'}^{\flat}([p_{g(\mu')\lambda}(x), t])= q_{\mu\mu'}\circ g'_{\mu' p_{g(\mu')\lambda}(x)}(t).$$
Also,
$$\mathcal{S}(g'_{\mu p_{g(\mu)\lambda}(x)})= g_{\mu}( p_{g(\mu)\lambda}(x))$$
and
$$\mathcal{S}(q_{\mu\mu'}\circ g'_{\mu' p_{g(\mu)\lambda}(x)})= q_{\mu\mu'}\circ g_{\mu'}( p_{g(\mu')\lambda}(x)).$$
Hence, using (\ref{hom3}) and \cite[Theorem 1.2.4]{CM},
$$g'_{\mu p_{g(\mu)\lambda}(x)}\simeq q_{\mu\mu'}\circ g'_{\mu' p_{g(\mu)\lambda}(x)} \ \ (rel \{\dot{I}\})$$
and so
$$ g_{\mu}^{\flat} \circ \Sigma p_{g(\mu)\lambda} \simeq q_{\mu\mu'}\circ g_{\mu'}^{\flat} \circ \Sigma p_{g(\mu')\lambda} \ \ (rel\{\Sigma x_{\lambda}\}).$$
\end{proof}

Let $\widetilde{Sh}_*$ be a subcategory of Sh$_*$ consists of all pointed topological spaces having bi-expansions. In follow, we conclude some results in the subcategory $\widetilde{Sh}_*$. It is well-known that the pair $(\Sigma , \Omega)$ is an adjoint pair of functors on hTop$_*$. In the following theorem we prove similar result on subcategory $\widetilde{Sh}_*$.
\begin{theorem}\label{adj}
For every $(Y,y)\in \widetilde{Sh}_*$ and $(X,x)\in Sh_*$, there is a natural bijection
\begin{equation}\label{eq}
Sh (\Sigma (X, x), (Y,y))\cong Sh((X,x),(Sh((I, \dot{I}),(Y,y)),e_y)).
\end{equation}
\end{theorem}
\begin{proof}
Let $\mathbf{p}:(X,x)\rightarrow (\mathbf{X},\mathbf{x})$ be an HPol$_*$-expansion of $(X,x)$ and $\mathbf{q}:(Y,y)\rightarrow (\mathbf{Y},\mathbf{y})$  be a bi-expansion of $(Y,y)$. We define
$$\tau_{XY}: Sh (\Sigma (X, x), (Y,y))\rightarrow Sh((X,x),(Sh((I, \dot{I}),(Y,y)),e_y))$$ by $\tau_{XY}(F)=F^{\sharp}$ and
\[\theta_{XY}:Sh((X,x),(Sh((I, \dot{I}),(Y,y)),e_y))\rightarrow Sh (\Sigma (X, x), (Y,y))\] by $\theta_{XY}(G)=G^{\flat}$. By Lemmas \ref{lem1} and \ref{lem2}, the maps $\tau_{XY}$ and $\theta_{XY}$ are well-defined. It is easy to see that $\theta_{XY}\circ\tau_{XY}=id$, $\tau_{XY}\circ\theta_{XY}=id$ and $\tau_{XY}$ is natural in each variable. Hence the result holds.
\end{proof}

Using natural bijection (\ref{eq}), one can see that the functor $Sh((I, \dot{I}),-)$ preserves inverse limits such as products, pullbacks, kernels, nested intersections and completions, provided inverse limit exists in the subcategory $\widetilde{Sh}_*$. Also, the functor $\Sigma$ preserves direct limits of connected spaces in this subcategory. Hence if $(X\times Y, (x,y))$ is a product of pointed spaces $(X, x)$ and $(Y,y)$ in the subcategory $\widetilde{Sh}_*$, then 
$$Sh((I, \dot{I}), (X\times Y, (x,y)))=Sh((I, \dot{I}), (X, x))\times Sh((I, \dot{I}), (Y,y))$$
and so
$$\check{\pi}_1(X\times Y, (x,y))=\check{\pi}_1(X, x)\times\check{\pi}_1(Y, y).$$
\begin{lemma}\label{con} 
The mappings $\tau_{XY}$ and $\theta_{XY}$ are continuous.
\end{lemma}
\begin{proof}
First, we show that $\tau_{XY}$ is continuous. Let $V^F_{\mu}$ be a basis element of \linebreak $Sh((X,x),(Sh((I, \dot{I}),(Y,y)),e_y))$ containing $F$. We will show that $\tau_{XY}(V^{F^\flat}_{\mu})\subseteq V^F_{\mu}$. Let $G\in V^{F{^\flat}}_{\mu}$. By definition, $q_{\mu}\circ F^\flat= q_{\mu}\circ G$ as homotopy classes to $Y_{\mu}$, or equivalently $f_{\mu}^{\flat}\circ \Sigma p_{f(\mu)}\simeq g_{\mu}\circ \Sigma p_{g(\mu)} \ \ (rel\{\Sigma x\}) $. It is sufficient to show that $(q_{\mu})_*\circ F= (q_{\mu})_*\circ G^{\sharp}$ as homotopy classes to $Sh(I,Y_{\mu})$ or equivalently $f_{\mu}\circ p_{f(\mu)}\simeq g_{\mu}^{\sharp}\circ p_{g(\mu)} \ \ (rel\{x\})$. For every $x\in X$,
$$g_{\mu}^{\sharp}\circ p_{g(\mu)} (x)=\mathcal{S}(l_{p_{g(\mu)} (x)\mu}),$$
and for every $t\in I$,
$$l_{p_{g(\mu)}(x)\mu}(t) = g_{\mu}([p_{g(\mu)}(x), t])= g_{\mu} \circ \Sigma p_{g(\mu)}([x,t]).$$
Also
\begin{align}
f_{\mu}^{\flat}\circ \Sigma p_{f(\mu)} ([x,t])&= f_{\mu}^{\flat}([p_{f(\mu)}(x), t]) \nonumber \\
& = f'_{\mu p_{f(\mu)}(x)}(t), \nonumber
\end{align}
where $\mathcal{S}(f'_{\mu p_{f(\mu)}(x)})=f_{\mu}(p_{f(\mu)}(x))$.
Since $f_{\mu}^{\flat}\circ \Sigma p_{f(\mu)}\simeq g_{\mu}\circ \Sigma p_{g(\mu)}\ \  ( rel\{\Sigma x\})$, by the above equalities, $l_{p_{g(\mu)}(x)\mu}\simeq f'_{\mu p_{f(\mu)}(x)}\ \ ( rel\{\dot{I}\})$.
Thus
$$g_{\mu}^{\sharp}\circ p_{g(\mu)} (x)=\mathcal{S}(l_{p_{g(\mu)}(x)\mu}) = \mathcal{S}(f'_{\mu p_{f(\mu)}(x)})=f_{\mu}(p_{f(\mu)}(x)).$$
So $\tau_{XY}(G)=G^{\sharp}\in V^F_{\mu}$, and therefore $\tau_{XY}$ is continuous. Similarly, $\theta_{XY}$ is continuous.
\end{proof}

In particular, we can conclude that for any pointed topological space $(X,x)$, 
\linebreak $Sh((I, \dot{I}),(Sh((I, \dot{I}),(X,x)),e_x))\cong Sh((I^2, \dot{I^2}),(X,x))$. We know that for any pointed space $(X,x)$ and for all $1\leq k\leq n-1$, $\pi_n(X,x)\cong \pi_{n-k}(\Omega(X,x), e_x)$. As a result of Theorem \ref{adj}, we have the following corollary:
\begin{corollary}\label{col}
Let $(X,x)$ be a pointed topological space in $\widetilde{Sh}_*$. Then for all $1\leq k \leq n-1$
$$\check{\pi}_n(X,x)\cong \check{\pi}_{n-k}(Sh((S^k, *),(X,x)), e_x).$$
\end{corollary}
\begin{proof}
By the definition of the shape homotopy group and using Theorem \ref{adj} and Lemma \ref{con}, we have
\begin{align}
 \check{\pi}_n(X,x)& = Sh((S^n, *),(X,x))\cong Sh((\Sigma{} ^{n}S^0, *),(X,x)) \nonumber \\
& \cong Sh((\Sigma{} ^{n-k}S^0, *),(Sh((S^k, *),(X,x)),e_x)) \nonumber \\
& \cong  Sh((S^{n-k}, *),(Sh((S^k, *),(X,x)), e_x)) \nonumber \\
& =\check{\pi}_{n-k}(Sh((S^k, *),(X,x)), e_x). \nonumber
\end{align}
\end{proof}

In follow, we exhibit an example in which the above corollary and therefore Theorem \ref{adj} do not hold in the category  Sh$_*$. 
\begin{remark}
The pair $(\Sigma, Sh((I, \dot{I}),-))$ is not an adjoint pair of functors on the category Sh$_*$. By contrary, if the pair $(\Sigma, Sh((I, \dot{I}),-))$ is an adjoint pair on Sh$_*$, with the same argument we obtain  $\check{\pi}_n(X,x)\cong \check{\pi}_{n-k}(Sh((S^k, *),(X,x)), e_x)$, for all $1\leq k \leq n-1$ and for all pointed topological space $(X,x)$. But this isomorphism does not hold in general. Put $ X=S^2$ and $n=2$, we have $\check{\pi}_2(S^2)= \pi_2(S^2)=\Bbb{Z}$ while $\check{\pi}_{1}(Sh(S^1,S^2))$ is trivial. Note that, $S^2$ is a polyhedron and so $Sh(S^1,S^2)$ is discrete by \cite[Theorem 4.4]{Na}. Hence $\check{\pi}_{1}(Sh(S^1,S^2))$ is trivial.

\end{remark}

Nasri et al. in \cite{Na2} showed that  for any pointed topological space $(X,x)$, $\pi_n^{qtop}(X,x)\cong \pi_{n-k}^{qtop}(\Omega^k(X,x), e_x)$, for all $1\leq k \leq n-1$. In the following corollary we prove this result for $\check{\pi}_n^{top}$. The following result is an immediate consequence of Corollary \ref{col} and Lemma \ref{con}.
\begin{corollary}
Let $(X,x)$ be a pointed topological space in $\widetilde{Sh}_*$. Then for all $1\leq k \leq n-1$
$$\check{\pi}_n^{top}(X,x)\cong \check{\pi}_{n-k}^{top}(Sh((S^k, *),(X,x)), e_x).$$
\end{corollary}
\section*{Acknowledgements}
This research was supported by a grant from Ferdowsi University of Mashhad-Graduate Studies (No. 2/43171).







\section*{References}

\bibliography{mybibfile}

\end{document}